\documentclass[11pt]{scrartcl}

\usepackage{amssymb,amsmath,amsthm,amsrefs,braket,authblk}
\usepackage[dvipdfmx]{graphicx}

\newtheorem{Theorem}{\textbf{Theorem}}[section]
\newtheorem{Lemma}[Theorem]{\textbf{Lemma}} 
\newtheorem{Proposition}[Theorem]{\textbf{Proposition}} 
\newtheorem{Remark}[Theorem]{\textbf{Remark}}
\newtheorem{Example}[Theorem]{\textbf{Example}}

\newenvironment{thm}{\begin{Theorem}}{\end{Theorem}}
\newenvironment{lem}{\begin{Lemma}}{\end{Lemma}}
\newenvironment{prop}{\begin{Proposition}}{\end{Proposition}}
\newenvironment{rem}{\begin{Remark}}{\end{Remark}}
\newenvironment{example}{\begin{Example}}{\end{Example}}

\makeatletter
    
    \@addtoreset{equation}{section}
  \makeatother

\newcommand{\RR}{\mathbb{R}}
\newcommand{\al}{\alpha}
\newcommand{\ga}{\gamma}
\newcommand{\ep}{\varepsilon}
\newcommand{\ran}{\rangle}
\newcommand{\lan}{\langle}

\title{Stability conditions of an ODE arising in human motion and its numerical simulation}  
\author[1]{Takahiro Kosugi\thanks{T. Kosugi is supported by JSPS KAKENHI Grant Number JP18K13436 and MEXT-Supported Program for the Strategic Research Foundation at Private Universities, Japan.
}}
\author[1]{Hitoshi Kino}
\author[1]{Masaaki Goto}
\author[2]{\\ Yuki Matsutani}

\affil[1]{\small Department of Intelligent Mechanical Engineering, Faculty of Engineering,

Fukuoka Institute of Technology, Fukuoka 811-0295, Japan}

\affil[2]{\small Department of Robotics, Faculty of Engineering, 

Kindai University, Higashi-Hiroshima 739-2116, Japan}

\date{}

\begin{document}
\maketitle
\begin{abstract}
This paper discusses the stability of an equilibrium point of an ordinary differential equation (ODE) arising from a feed-forward position control for a musculoskeletal system. The studied system has a link, a joint and two muscles with routing points.
The motion convergence of the system strongly depends on the muscular arrangement of the musculoskeletal system. In this paper, a sufficient condition for asymptotic stability is obtained. Furthermore, numerical simulations of the penalized ODE and experimental results are described. 
\end{abstract}


\tableofcontents

\section{Introduction}

The mechanism of human motion is expected to be applied to robotics.
Some hypotheses such as ``the equilibrium point hypothesis'' \cite{Feldman, EP_science} and ``the virtual trajectory hypothesis'' \cite{Hogan}  suggest that human motion generation efficiently utilizes a feed-forward position control.
In addition, rapid motion without  sensory feedback, such as a finger flick, can be controlled to some extent. 
Therefore, it is assumed that a human musculoskeletal system satisfies the controllability condition for a feed-forward control.
We would like to interpret that assumption in both a mathematical and engineering sense.

The approach of our feed-forward control is to keep muscular tensions balanced at a target position.
In 2013, Kino et al. \cite{KKMTN13} studied the feed-forward control for a 2-link-6-muscle musculoskeletal system, which is modeled after a human arm.
They gave an engineering consideration and showed their numerical simulations.
We also refer to Kino et al. \cite{KOMT17} for a mathematically sufficient condition of the same system as \cite{KKMTN13} for feed-forward control. 
However, the system is considered without some characteristics.
For instance, each muscle of the system is arranged as straight, even though complicated arrangements actually exist around the joints.

In this paper,  we consider a musculoskeletal system  (Figure \ref{fig:target_system}) that has a link, a joint and two muscles. 
This system is modeled after a human finger.  
The name of each part is defined according to Figure \ref{fig:target_system}. 
We also use the same symbols for the lengths. 
Base $L_0$ is always fixed. 
Link $L_1$ rotates on the joint in the two-dimensional plane. 
The red line and the blue line in Figure \ref{fig:target_system} are muscles 
and are called Muscle 1 and Muscle 2, respectively.
These muscles have routing points $P_{12}$ and $P_{22}$. Generally, human muscles are restricted by the tendons and curve with the rotation of the links. 
For simplification, in this system, we assume that Muscle 1 and Muscle 2 bend at routing points $P_{12}$ and $P_{22}$  (Figure \ref{fig:routingpoint}), which are the lengths $\ell_1$ and $\ell_2$ away from the joint, respectively. 
We call $\ell_1$ and $\ell_2$ virtual links (Figure \ref{fig:virtuallink}), although the system primarily has the unique link, Link $L_1$.  
We let the virtual links $\ell_1$ and $\ell_2$ rotate by half of the rotation angle of Link $L_1$. 
We suppose that friction of the joint 
comprises only viscosity friction with a viscosity friction coefficient of $\mu>0$ and that the system ignores a viscoelasticity of muscles.
Muscles are imaged to resemble wires.
We also ignore gravity effects.
In this case, we are concerned with a sufficient condition for the above feed-forward control of the system.

\begin{figure}[h!]
	\centering
		\includegraphics[keepaspectratio=true,height=54mm]{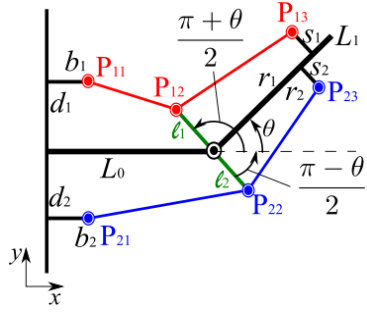}\vspace{-1.0mm}
		\caption{1-link-2-muscle musculoskeletal system with routing points}
		\label{fig:target_system}
\end{figure}
\begin{figure}[h!]
\begin{tabular}{c}
 \begin{minipage}{0.5\hsize}
  \begin{center}
   \includegraphics[
scale=0.4
   ]{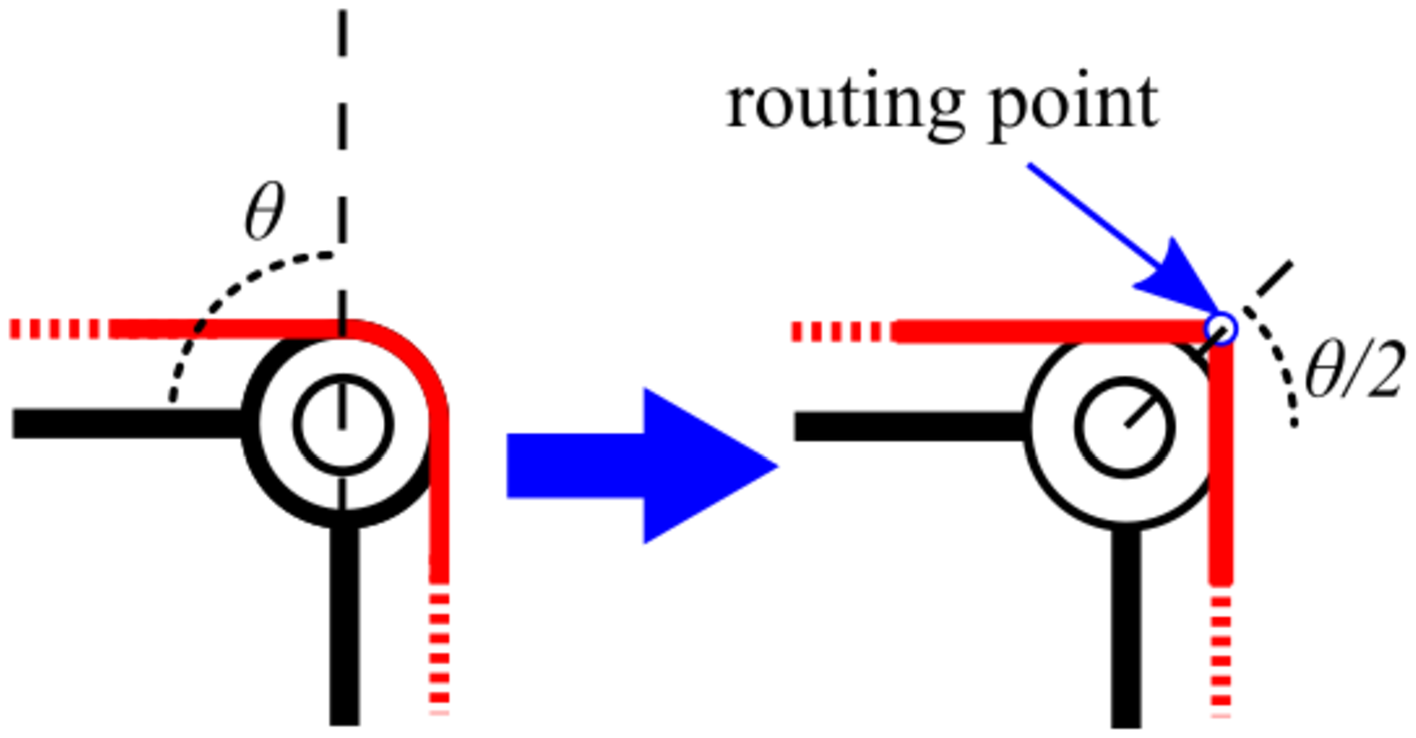}
  \end{center}
  \caption{Routing point}
  \label{fig:routingpoint}
 \end{minipage}
 \begin{minipage}{0.5\hsize}
  \begin{center}
   \includegraphics[
scale=0.5
   ]{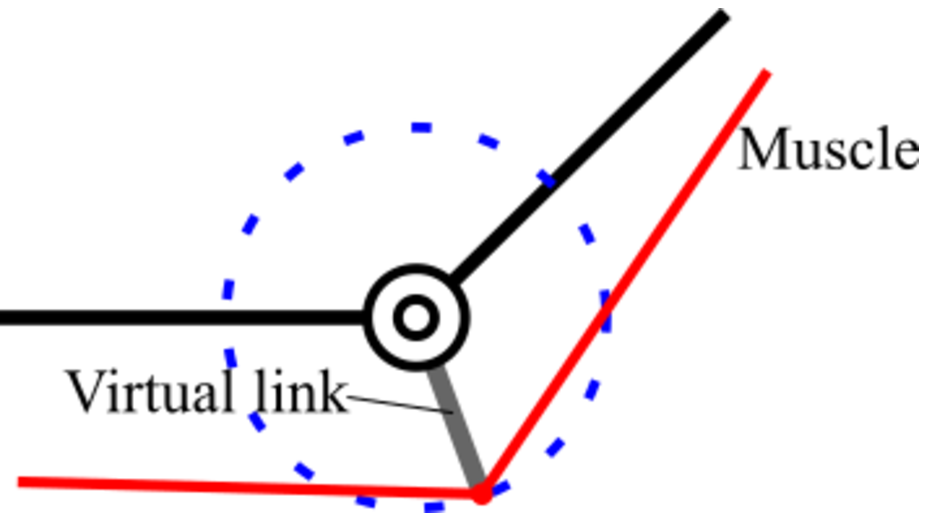}
  \end{center}
  \caption{Virtual link}
  \label{fig:virtuallink}
 \end{minipage}
 \end{tabular}
\end{figure}
To this end, as a mathematical problem, this paper discusses a sufficient condition for asymptotic stability of a equilibrium point of the dynamics
\begin{equation}\label{introeq}
I\ddot\theta(t) +\mu\dot\theta(t)-\tau(\theta(t))=0,\quad t>0,
\end{equation}
where $\theta:(0,\infty)\to (-\theta_0,\pi)$ is an unknown function. 
Here, a constant $\theta_0>0$ is small enough. A constant $I$, a constant $\mu$ and a function $\tau:(-\theta_0,\pi)\to\RR$ are given. 
$I$ is the moment of inertia, $\mu$ is the viscosity coefficient of the joint, and $\tau$ is the torque of the link generated by constant muscular tensions balancing at the target position $\theta=\theta_d$. 

We give a derivation of $\tau$. Let $F_1$ and $F_2$ be constant muscular tensions of Muscle 1 and Muscle 2, respectively. 
Let $q_{ij}=q_{ij}(\theta):\RR\to\RR$ ($i,j=1,2$) be the length of $P_{ij}P_{i(j+1)}$ as shown in Figure \ref{fig:target_system}.
They are given by
\[
\begin{split}
&q_{11}(\theta)=\left\{\left(L_0-b_1-\ell_1\sin\dfrac{\theta}{2}\right)^2+\left(d_1-\ell_1\cos\dfrac{\theta}{2}\right)^2\right\}^{1/2},\\
&q_{12}(\theta)=\left\{\left(r_1-\ell_1\sin\dfrac{\theta}{2}\right)^2+\left(s_1-\ell_1\cos\dfrac{\theta}{2}\right)^2\right\}^{1/2},\\
&q_{21}(\theta)=\left\{\left(L_0-b_2+\ell_2\cos\left(\pi+\dfrac{\pi+\theta}{2}\right)\right)^2+\left(d_2+\ell_2\sin\left(\pi+\dfrac{\pi+\theta}{2}\right)\right)^2\right\}^{1/2},\\
&q_{22}(\theta)=\left\{\left(r_2+\ell_2\cos\dfrac{\pi-\theta}{2}\right)^2+\left(s_2-\ell_2\sin\dfrac{\pi-\theta}{2}\right)^2\right\}^{1/2},
\end{split}
\]
where parameters $L_0,b_1,b_2,d_1,d_2,\ell_1,\ell_2,r_1,r_2,s_1$ and $s_2$ are positive constants.
Here, $q_{i}:=q_{i1}+q_{i2}$ denotes the length of Muscle $i$ for $i=1,2$. 
By the principle of virtual work, we have
\[
\tau(\theta)=-\left\lan J(\theta),\begin{pmatrix}F_1\\ F_2\end{pmatrix}\right\ran.
\]
Here, $\lan\cdot,\cdot\ran:\RR^n\times\RR^n\to\RR$ denotes the Euclidean inner product. 
It follows that
\[
\RR^2\ni F:=\begin{pmatrix}F_1\\ F_2\end{pmatrix} =-\tau(\theta)J(\theta)^*+v(\theta),
\]
where  $q:=(q_1,q_2)$, $J(\theta):=\dfrac{dq}{d\theta}(\theta)$ and
\[
J(\theta)^*:=
\begin{cases}
\dfrac{J(\theta)}{|J(\theta)|^2}\quad &\mbox{if }J(\theta)\neq0,\\
0&\mbox{if }J(\theta)=0.
\end{cases}
\]
The second term $v(\theta)$ of the right-hand side belongs to $\Set{v\in\RR^2 | \lan J(\theta), v\ran =0}$, namely, 
\[
v(\theta)=k\begin{pmatrix} \dfrac{dq_2}{d\theta}(\theta)\\[10pt] -\dfrac{dq_1}{d\theta}(\theta)\end{pmatrix} \quad\mbox{for any } k\in\RR.
\]
Here, the first term $-\tau(\theta)J(\theta)^*$ denotes a force to rotate the link $L_1$, that is, a driving force at $\theta$. 
The second term $v(\theta)$ denotes an orthogonal force to the rotation of $L_1$, that is, an inner force at $\theta$. 
For instance, $F=v(\theta_d)$ is an internal force vector of Muscle 1 and Muscle 2 balancing at $\theta=\theta_d$.
Since our feed-forward control is to take $F\equiv v(\theta_d)$, the torque $\tau(\theta)$ is given by
\[
\tau(\theta)=-\left\lan \dfrac{dq}{d\theta}(\theta),v(\theta_d) \right\ran
=-\left\lan \begin{pmatrix} \dfrac{dq_1}{d\theta}(\theta)\\[10pt] \dfrac{dq_2}{d\theta}(\theta)\end{pmatrix},
k\begin{pmatrix} \dfrac{dq_2}{d\theta}(\theta_d)\\[10pt] -\dfrac{dq_1}{d\theta}(\theta_d)\end{pmatrix}\right\ran
\]
for some $k\in\RR$. The dynamics of \eqref{introeq} becomes the ordinary differential equation 
\begin{equation}\label{meq}
I\ddot\theta +\mu\dot\theta+\left\lan \dfrac{dq}{d\theta}(\theta),v(\theta_d) \right\ran=0,\quad t>0.
\end{equation}
We note that the solution, $\theta=\theta_d$, is an equilibrium point of \eqref{meq}.

Our aim in this paper is to give conditions for parameters $L_0$, $b_1$, $b_2$, $d_1$, $d_2$, $\ell_1$, $\ell_2$, $r_1$, $r_2$, $s_1$ and $s_2$, and an equilibrium point $\theta_d$, for the asymptotic stability of $\theta=\theta_d$. 
Since each muscle of the 2-link-6-muscle musculoskeletal system in \cite{KOMT17} is straight, 
they show a sufficient condition by applying the Taylor expansion of the muscular lengths at some angle. 
In this paper, we use a different method because of a complicated $q_i$ by routing points.

This paper is organized in the following way.
In Section \ref{sec:preliminary}, we recall some known results. 
Sections \ref{sec:scond} and \ref{sec:proof} are devoted to showing a sufficient condition.
In Section \ref{sec:nande}, we show the numerical simulation results and experimental results.

\section{Preliminaries}\label{sec:preliminary}

We first recall Lyapunov's stability theorem, which corresponds to that of Theorem 1.30 in \cite{C}.
We also refer to \cite{A} for an application to the control of mechanical systems.

\begin{prop}[Theorem 1.30 in \cite{C}]\label{prop:Lyapunov}
Let $x_0$ be an equilibrium point of the autonomous ordinary differential equation
\begin{align}\label{eq:pre_ode}
\dot x(t)=g(x(t)), \quad t>0.
\end{align}
Let continuous function $V:U\to\RR$ be a Lyapunov function for \eqref{eq:pre_ode} at $x_0$, i.e., $V$ implies the following:
\begin{itemize}
\item $V(x_0)=0$;
\item $V(x)>0$ for $x\in U\setminus\{x_0\}$;
\item the function $x\mapsto \mathrm{grad}\, V(x)$ is continuous for $x\in U\setminus\{x_0\}$, and on this set, $\dfrac{d}{dt}(V(x(t)))\leq0$,
\end{itemize}
where $U\subset\RR^n$ is an open set. 
Then, $x_0$ is Lyapunov stable. 
In addition, if $V$ is a strict Lyapunov function, i.e., $\dfrac{d}{dt}(V(x(t)))<0$ for $x\in U\setminus\{x_0\}$, 
then $x_0$ is asymptotically stable.
\end{prop}

Define a Lagrangian $L_0: \RR^n\times \RR^n\to \RR$  by
\begin{equation*}
L_0(x,\xi):=K(\xi)-P(x).
\end{equation*}
Here, $K:\RR^n\to\RR$ and $P:\RR^n\to \RR$ are respectively a kinetic energy and a potential energy. 
We consider the Lagrange equation
\begin{equation}\label{gLeq}
\dfrac{\partial^2 K}{\partial\xi^2}(\dot x(t))\ddot{x}(t)+\dfrac{\partial P}{\partial x}(x(t))=Q(t,x(t),\dot x(t))\quad\mbox{for }t>0,
\end{equation}
where $Q:(0,\infty)\times\RR^n\times \RR^n\to\RR^n$ is a generalized force. We always assume that $K$, $P$ and $Q$ are smooth functions.

We next recall an application of Lyapunov's stability theorem to \eqref{gLeq}. 
The following proposition provides a sufficient condition for Lyapunov/asymptotic stability of  an equilibrium  of \eqref{gLeq}. 
We write the proof for researchers in other fields, although the proof is elementary.

\begin{prop}\label{prop:stablecond}
Let $x=x_d\in\RR^n$ be an equilibrium point of \eqref{gLeq}. 
Let $K$ be convex such that $K(0)=0$.
Let $Q$ satisfy
\begin{equation}\label{gfc1}
\lan Q(t,x,\xi),\xi\ran\leq 0\quad\mbox{for } t>0,\ x\in\RR^n,\ \xi\in\RR^n.
\end{equation}
Assuming that $P$ is locally positive definite around $x_d$, i.e., there exists $R>0$ such that
\begin{equation}\label{pp1}
P(x_d)=0\quad\mbox{and}\quad P(x)>0\quad\mbox{for }x\in B_R(x_d)\setminus \{x_d\}.
\end{equation}
Then, $x=x_d$ is Lyapunov stable. Furthermore, if
\begin{equation}\label{gfc2}
\lan Q(t,x,\xi),\xi\ran < 0\quad\mbox{for } t>0,\ (x,\xi)\in\RR^n\times \RR^n\setminus\{(x_d,0)\},
\end{equation}
then $x=x_d$ is an asymptotically stable equilibrium point.
\end{prop}

\begin{proof}
We set 
\[
V(x,\xi):=\left\lan\dfrac{\partial K}{\partial \xi}(\xi),\xi\right\ran-K(\xi)+P(x).
\]
We show that $V$ becomes a Lyapunov function for
\begin{align}\label{gLs}
\begin{cases}
&\dfrac{\partial^2 K}{\partial\xi^2}(\xi(t))\dot{\xi}(t)+\dfrac{\partial P}{\partial x}(x(t))=Q(t,x(t),\xi(t)),\\
&\xi(t)=\dot x(t),
\end{cases}
\quad t>0.
\end{align}
at an equilibrium point, $(x,\xi)=(x_d,0)$.
Since $K$ is convex and \eqref{pp1} holds, $V$ is locally positive definite around $(x_d,0)$.

By \eqref{gLeq} and \eqref{gfc1}, we have
\begin{align*}
\dfrac{d}{dt}(V(x(t),\xi(t)))=\left\lan \dfrac{\partial^2K}{\partial\xi^2}(\xi)\dot\xi,x\right\ran+\left\lan\dfrac{\partial P}{\partial x},\xi \right\ran
=\lan Q,\xi\ran
\leq 0.
\end{align*}
Thanks to Proposition \ref{prop:Lyapunov}, $x_d$ is Lyapunov stable since $V$ is a Lyapunov function for \eqref{gLs}.

In addition, we assume \eqref{gfc2}. Similarly, it follows that
\[
\dfrac{d}{dt}(V(x(t),\xi(t)))
< 0\quad\mbox{for } t>0,\ (x,\xi)\in\RR^n\times \RR^n\setminus\{(x_d,0)\},
\]
which implies that $x_d$ is asymptotically stable.
\end{proof}

\section{Stability condition of a 1-link-2-muscle musculoskeletal system with routing points}\label{sec:scond}

In this section, we are concerned with a sufficient condition for the stability of the equilibrium point, $\theta=\theta_d$, of the equation \eqref{meq}.

Define a function $f=f(\cdot;i,j):(-\theta_0,\pi)\to\RR$ by
\begin{align}
\begin{split}\label{fdef}
f(\theta)=f(\theta;i,j)&:=\left\{\left(a_{ij}+b_{ij}\sin\dfrac{\theta}{2}\right)^2+\left(c_{ij}+b_{ij}\cos\dfrac{\theta}{2}\right)^2\right\}^{1/2}\\
&=\left\{a_{ij}^2+b_{ij}^2+c_{ij}^2+2b_{ij}\sqrt{a_{ij}^2+c_{ij}^2}\sin\left(\dfrac{\theta}{2}+\alpha_{ij}\right)\right\}^{1/2}
\end{split}
\end{align}
for constants $a_{ij},b_{ij},c_{ij}\in\RR$, where $\alpha_{ij}\in[0,2\pi)$ satisfies
\begin{equation*}
\sin\al_{ij}=\dfrac{c_{ij}}{\sqrt{a_{ij}^2+c_{ij}^2}},\quad \cos\al_{ij}=\dfrac{a_{ij}}{\sqrt{a_{ij}^2+c_{ij}^2}}.
\end{equation*}
In what follows, the constants $a_{ij},b_{ij},c_{ij}$ are always given by
\begin{align}\label{eq:defaijbijcij}
\begin{split}
\begin{array}{llr}
&a_{11}=L_0-b_1, \ b_{11}=-\ell_1, \ c_{11}=d_1\qquad &\mbox{for }i=1,\ j=1; \\
&a_{12}=r_1,\ b_{12}=-\ell_1, \ c_{12}=s_1&\mbox{for }i=1,\ j=2; \\
&a_{21}=-(L_0-b_2),\ b_{21}=-\ell_2,\ c_{21}=d_2&\mbox{for }i=2,\ j=1; \\
&a_{22}=-r_2,\ b_{22}=-\ell_2,\ c_{22}=s_2&\mbox{for }i=2,\ j=2. 
\end{array}
\end{split}
\end{align}
We note that $f(\theta;i,j)$ becomes $q_{ij}$ ($i,j=1,2$).
We compute $f=f(\cdot;i,j)$ instead of $q_{ij}$ to obtain a sufficient condition for the stability.
We note that 
\begin{align}\label{cond:bc}
b_{ij}<0\quad\mbox{and}\quad c_{ij}>0\quad\mbox{for }i,j=1,2.
\end{align}

We present a list of hypotheses on $a_{ij},b_{ij},c_{ij}$ and $\theta_0$. 
We assume the following for $f>0$:
\begin{align}\label{eq:assum1}
\sqrt{a_{ij}^2+c_{ij}^2}\neq |b_{ij}|.
\end{align}
This assumption means that $P_{ij}\neq P_{i(j+1)}$ ($i,j=1,2$) as shown in Figure \ref{fig:target_system}.

Next, we suppose that 
\begin{equation}\label{eq:assum2}
a_{1j}>0\quad \mbox{and}\quad a_{2j}<0\quad\mbox{for }j=1,2.
\end{equation}
Thus, we have
\begin{align}\label{cond:al}
\al_{11},\al_{12}\in (0,\pi/2)\quad \mbox{and}\quad \al_{21},\al_{22}\in (\pi/2,\pi).
\end{align}

Finally, we give an assumption of $\theta_0>0$. We set
\[
C_{\theta_0}=C_{\theta_0,ij}:=\sqrt{a_{ij}^2+c_{ij}^2}\sin\left(-\dfrac{\theta_0}{2}+\al_{ij}\right).
\]
To have a comparison between $C_{\theta_0}$, $a$ and $c$, we assume that
\begin{equation}\label{theta01}
0<\theta_0<\min\{2\min\{\al_{21},\al_{22}\}-\pi,\ \pi/4\}.
\end{equation}
In fact, \eqref{theta01} implies
\[
\dfrac{c_{2j}}{\sqrt{a_{2j}^2+c_{2j}^2}}=\sin\al_{2j}<\sup_{\theta\in(-\theta_0,\pi)}\sin\left(\dfrac{\theta}{2}+\al_{2j}\right)=\sin\left(-\dfrac{\theta_0}{2}+\al_{2j}\right)=\dfrac{C_{\theta_0,2j}}{\sqrt{a_{2j}^2+c_{2j}^2}}<1.
\]

We assume that $k\dfrac{dq_2}{d\theta}(\theta_d)$ and $ -k\dfrac{dq_1}{d\theta}(\theta_d)$ in $v(\theta_d)$ are  muscular tensions. Therefore, we require that they are positive. 
A sufficient condition for that requirement is as follows.

\begin{prop}\label{prop:ifp}
Assume that \eqref{eq:assum1}, \eqref{eq:assum2} and \eqref{theta01} hold. 
We have $\dfrac{dq_2}{d\theta}(\theta_d)>0$ for $\theta_d\in(-\theta_0,\pi)$.
In addition, we assume
\begin{equation}\label{ifpcond}
-\theta_0<\theta_d<\pi-2\max\{\al_{11},\al_{12}\}.
\end{equation}
Then, $k\dfrac{dq_2}{d\theta}(\theta_d)$ and $ -k\dfrac{dq_1}{d\theta}(\theta_d)$ are positive for any $k>0$.
\end{prop}

\begin{rem}
For a typical 1-link-2-muscle muscloskeletal system without routing points, 
balanced muscular tensions
are always positive at any target angle.
\end{rem}

\begin{proof}[Proof of Proposition \ref{prop:ifp}]
From \eqref{fdef}, we have
\begin{equation*}
f'(\theta_d)=\dfrac{b_{ij}\sqrt{a_{ij}^2+c_{ij}^2}}{2}f(\theta_d)^{-1}\cos\left(\dfrac{\theta_d}{2}+\al_{ij}\right).
\end{equation*}
By \eqref{cond:al}, it is obvious that
\[
\cos\left(\dfrac{\theta_d}{2}+\al_{2j}\right)<0 \quad\mbox{for } \theta_d\in(-\theta_0,\pi),\ j=1,2,
\]
which implies that $f'(\theta_d;2,j)>0$. Thus, $\dfrac{dq_2}{d\theta}(\theta_d)=\sum_{j=1}^2f'(\theta_d;2,j)>0$.

Let us prove that $-\dfrac{dq_1}{d\theta}(\theta_d)>0$ under \eqref{ifpcond}. By \eqref{cond:al} and \eqref{theta01}, it follows that
\[
-\dfrac{\pi}{8}<\dfrac{\theta_d}{2}+\al_{1j}<\dfrac{\pi}{2},
\]
which yields
\[
\cos\left(\dfrac{\theta_d}{2}+\al_{1j}\right)>0 \quad\mbox{for }j=1,2.
\]
Thus, $f'(\theta_d;1,j)<0$ for $j=1,2$, which implies $-\dfrac{dq_1}{d\theta}(\theta_d)=-\sum_{j=1}^2f'(\theta_d;1,j)>0$.
\end{proof}

The following theorem asserts that there exists a sufficient condition for the asymptotic stability of $\theta=\theta_d$.

\begin{thm}\label{thm:sc}
Assume that \eqref{eq:assum1}, \eqref{eq:assum2} and \eqref{theta01} hold.
Then, for $a_{ij}, b_{ij},c_{ij}$ given by \eqref{eq:defaijbijcij}, 
the equilibrium point $\theta_d\in\bigcup_{i,j=1,2}\Theta_{ij}\cap\Theta_0$  of \eqref{meq} is asymptotically stable, where $\Theta_{ij}$ and $\Theta_0$ are defined by 
\begin{equation*}
\begin{split}
&\Theta_0:=\Set{\theta | \theta \mbox{ satisfies }\eqref{ifpcond}},\\
&\Theta_{1j}:=\Set{\theta | \theta \mbox{ satisfies } \eqref{eq:lem41}}\ \ \mbox{and} \ \
\Theta_{2j}:=\Set{\theta | \theta \mbox{ satisfies } \eqref{eq:lem42}}.
\end{split}
\end{equation*}
\end{thm}

\section{Proof of Theorem \ref{thm:sc}}\label{sec:proof}

By letting
\[
K(\xi):=\dfrac{1}{2}I\xi^2,\quad
P(\theta):=\lan q(\theta)-q(\theta_d),v(\theta_d)\ran\quad\mbox{and}\quad
Q(\xi):=-\mu\xi^2,
\]
\eqref{gLeq} becomes \eqref{meq}.
From Proposition \ref{prop:stablecond}, we see that $\theta_d$ is an asymptotically stable point if $P$ takes strictly the minimum at $\theta_d$.
By the definition of $v(\theta_d)$,
\[
\dfrac{dP}{d\theta}(\theta_d)=\left\lan \dfrac{dq}{d\theta}(\theta_d),v(\theta_d) \right\ran=0.
\]
By Proposition \ref{prop:ifp}, under \eqref{ifpcond}, if
\begin{equation*}
f''(\theta_d;i,j)=\dfrac{d^2q_{ij}}{d\theta^2}(\theta_d)>0\quad\mbox{for }i,j=1,2,
\end{equation*}
then
\begin{align*}
&\dfrac{d^2P}{d\theta^2}(\theta_d)
=k\Bigg(
\Bigg(\dfrac{d^2q_{11}}{d\theta^2}(\theta_d)+\dfrac{d^2q_{12}}{d\theta^2}(\theta_d) \Bigg)\dfrac{dq_2}{d\theta}(\theta_d)\\
&\hspace{30mm}
-\Bigg(\dfrac{d^2q_{21}}{d\theta^2}(\theta_d)+\dfrac{d^2q_{22}}{d\theta^2}(\theta_d) \Bigg)\dfrac{dq_1}{d\theta}(\theta_d)
\Bigg)>0
\end{align*}
for $k>0$. Hence, let us discuss the existence of $(\theta,a_{ij},b_{ij},c_{ij})$ that $f''(\theta)>0$ holds. 

Differentiating \eqref{fdef}, we have
\begin{align*}
f''(\theta)
=-\dfrac{b_{ij}\sqrt{a_{ij}^2+c_{ij}^2}}{4}f(\theta)^{-3}\left(A\sin^2\left(\dfrac{\theta}{2}+\al_{ij}\right)+B\sin\left(\dfrac{\theta}{2}+\al_{ij}\right)+A\right),
\end{align*}
where 
\[
A:=b_{ij}\sqrt{a_{ij}^2+c_{ij}^2}\quad\mbox{and}\quad B:=a_{ij}^2+b_{ij}^2+c_{ij}^2.
\]
By \eqref{cond:bc} and \eqref{eq:assum1}, the discriminant of a quadratic function implies
$f''(\theta)>0$ for $\theta\in(-\theta_0,\pi)$, satisfying
\begin{align}\label{eq:sinbe}
\dfrac{-B+\sqrt{B^2-4A^2}}{2A}<\sin\left(\dfrac{\theta}{2}+\al_{ij}\right)<\dfrac{-B-\sqrt{B^2-4A^2}}{2A}.
\end{align}
We observe that $\sqrt{B^2-4A^2}=|a_{ij}^2-b_{ij}^2+c_{ij}^2|$, which implies that
\begin{align}\label{eq:discriminant1}
\dfrac{-B+\sqrt{B^2-4A^2}}{2A}=\dfrac{|b_{ij}|}{\sqrt{a_{ij}^2+c_{ij}^2}}<1,\quad
\dfrac{-B-\sqrt{B^2-4A^2}}{2A}=\dfrac{\sqrt{a_{ij}^2+c_{ij}^2}}{|b_{ij}|}>1
\end{align}
in the case when $\sqrt{a_{ij}^2+c_{ij}^2}>|b_{ij}|$, and 
\[
\dfrac{-B+\sqrt{B^2-4A^2}}{2A}=\dfrac{\sqrt{a_{ij}^2+c_{ij}^2}}{|b_{ij}|}<1,\quad
 \dfrac{-B-\sqrt{B^2-4A^2}}{2A}=\dfrac{|b_{ij}|}{\sqrt{a_{ij}^2+c_{ij}^2}}>1
\]
in the case when $\sqrt{a_{ij}^2+c_{ij}^2}<|b_{ij}|$.

The following lemma shows a sufficient condition for $f''(\theta;1,j)>0$.

\begin{lem}\label{lem:1}
Let $i=1$. Assume that
\begin{align}\label{eq:lem41}
\begin{cases}
\max\left\{-\theta_0,\ 2\left(\ga_1-\al_{1j}\right)\right\}
<\theta
<\min\left\{\pi,\ 2\left(\pi-\ga_1-\al_{1j}\right)\right\}\quad &\mbox{if }\sqrt{a_{1j}^2+c_{1j}^2}>|b_{1j}|,\\
\max\left\{-\theta_0,\ 2\left(\ga_2-\al_{1j}\right)\right\}
<\theta
<\min\left\{\pi,\ 2\left(\pi-\ga_2-\al_{1j}\right)\right\} &\mbox{if }\sqrt{a_{1j}^2+c_{1j}^2}<|b_{1j}|,
\end{cases}
\end{align}
where $\sin^{-1}:[-1,1]\to[-\pi/2,\pi/2]$ is the inverse sine function, 
\[
\ga_1:=\sin^{-1}\left(\dfrac{|b_{1j}|}{\sqrt{a_{1j}^2+c_{1j}^2}}\right)\quad\mbox{and}\quad \ga_2:=\sin^{-1}\left(\dfrac{\sqrt{a_{1j}^2+c_{1j}^2}}{|b_{1j}|}\right).
\]
Then, $f''(\theta;1,j)>0$ for $j=1,2$. 
\end{lem}

\begin{proof}
We drop the indexes ``$1j$'' for simplicity.
Since $\al\in(0,\pi/2)$, 
\[
\sup_{\theta\in(-\theta_0,\pi)}\sin\left(\dfrac{\theta}{2}+\al\right)=1.
\]

We first suppose that $\sqrt{a^2+c^2}>|b|$. 
In view of \eqref{eq:sinbe} and \eqref{eq:discriminant1}, it follows that $f''(\theta)>0$ for $\theta\in(-\theta_0,\pi)$, satisfying
\begin{align}\label{eq:pf41_1}
\dfrac{|b|}{\sqrt{a^2+c^2}}<\sin\left(\dfrac{\theta}{2}+\al\right)\leq1.
\end{align}
By the definition of $\ga_1$, \eqref{eq:lem41} implies \eqref{eq:pf41_1}.

In the same manner, we can see that $f''>0$ in the case when $\sqrt{a^2+c^2}<|b|$, which completes the proof.
\end{proof}

We next show a sufficient condition for $f(\theta;2,j)>0$.
\begin{lem}\label{lem:2}
Let $i=2$. Let \eqref{theta01} hold. Assume that
\begin{align}\label{eq:lem42}
\begin{cases}
-\theta_0<\theta<2\left(\pi-\ga_3-\al_{2j}\right) \quad &\mbox{if }|b_{2j}|< \min\left\{\sqrt{a_{2j}^2+c_{2j}^2},\ C_{\theta_{0},2j}\right\},\\[5pt]
-\theta_0<\theta<2\left(\pi-\ga_4-\al_{2j}\right) &\mbox{if }|b_{2j}|>\sqrt{a_{2j}^2+c_{2j}^2}\max\left\{C_{\theta_{0},2j}^{-1}\sqrt{a_{2j}^2+c_{2j}^2},\ 1\right\},
\end{cases}
\end{align}
where
\[
\ga_3:=\sin^{-1}\left(\dfrac{|b_{2j}|}{\sqrt{a_{2j}^2+c_{2j}^2}}\right)\quad\mbox{and}\quad \ga_4:=\sin^{-1}\left(\dfrac{\sqrt{a_{2j}^2+c_{2j}^2}}{|b_{2j}|}\right).
\]
Then, $f''(\theta;2,j)>0$ for $j=1,2$.
\end{lem}

\begin{proof}
We drop the indexes ``$2j$'' for simplicity. 

We first assume that $|b|< \min\left\{\sqrt{a^2+c^2},\ C_{\theta_0}\right\}$.
By \eqref{theta01} and $\al\in(\pi/2,\pi)$, we have
\[
\inf_{\theta\in(-\theta_0,\pi)}\sin\left(\dfrac{\theta}{2}+\al\right)=\dfrac{a}{\sqrt{a^2+c^2}}<0,\ \ 
\sup_{\theta\in(-\theta_0,\pi)}\sin\left(\dfrac{\theta}{2}+\al\right)=\dfrac{C_{\theta_0}}{\sqrt{a^2+c^2}}\in(0,1).
\]
In view of \eqref{eq:sinbe}, it follows that $f''(\theta)>0$ for $\theta\in(-\theta_0,\pi)$, satisfying
\[
\dfrac{|b|}{\sqrt{a^2+c^2}}<\sin\left(\dfrac{\theta}{2}+\al\right)<\dfrac{C_{\theta_0}}{\sqrt{a^2+c^2}},
\]
which holds by \eqref{eq:lem42}.

We similarly conclude that  $f''(\theta)>0$ for
\[
-\theta_0<\theta<2\left(\pi-\ga_4-\al\right)
\]
in the case when $|b|>\sqrt{a^2+c^2}\max\left\{C_{\theta_{0}}^{-1}\sqrt{a^2+c^2},\ 1\right\}$.
\end{proof}

Thanks to Lemma \ref{lem:1} and \ref{lem:2}, it follows that $\theta_d\in\bigcup_{i,j=1,2}\Theta_{ij}\cap\Theta_0$ is an asymptotic stable point of \eqref{meq}. In fact, the following example is contained in the above.

\begin{example}\label{example1}
We take $L_0-b_1=2\kappa$, $\ell_1=\kappa$, $d_1=2\kappa$, $r_1=\dfrac{\kappa}{2\sqrt2}$, $s_1=\dfrac{\kappa}{2\sqrt2}$, $L_0-b_2=2\sqrt{2}\kappa$, $\ell_2=\kappa$, $d_2=2\sqrt{2}\kappa$, $r_2=\dfrac{\kappa}{2\sqrt2}$ and $s_2=\dfrac{\kappa}{2\sqrt2}$ for any $\kappa\in\RR$. Then, we have
\[
\al_{11}=\pi/4,\quad \al_{12}=\pi/4,\quad \al_{21}=3\pi/4,\quad \al_{22}=3\pi/4.
\]
Thanks to Proposition \ref{prop:ifp}, inner forces $v_{d1},v_{d2}>0$. Theorem \ref{thm:sc} shows that $\theta=\theta_d\in(0,\pi/6)$ is asymptotically stable.
\end{example}

\section{Numerical and experimental results}\label{sec:nande}
This section shows the numerical simulation results and experimental results in a stable case based on Example \ref{example1}. 
We also give an unstable case to show that the motion convergence of the system in Figure \ref{fig:target_system} is not always stable, although we did not discuss an unstable condition above.

\subsection{Numerical simulation}
Since the solution may be unstable, we use a penalized equation for \eqref{meq} in order to simulate.
\begin{align}\label{eq:peq}
I\ddot\theta+\mu\dot\theta+\left\lan \dfrac{dq}{d\theta}(\theta),v(\theta_d) \right\ran+\dfrac{(\theta-\theta_{\text{max}})^+}{\ep}-\dfrac{(\theta_{\text{min}}-\theta)^+}{\ep}=0, \quad t>0,
\end{align}
where $\theta_{\text{max}}=41\pi/180$, $\theta_{\text{min}}=-\pi/180$, $\ep=1/1000$ and
\[
r^+:=
\begin{cases}
r\quad (r\geq 0),\\
0\quad (r<0).
\end{cases}
\]

Let $\theta^\ep$ be a solution of \eqref{eq:peq}. We note that $\theta^\ep$ coincides with a solution of \eqref{meq} if it is always in $(\theta_{\text{min}},\theta_{\text{max}})$. 
We also note that $\theta^\ep$ is an approximation of a solution of \eqref{meq}  with $\theta_{\text{min}}\leq\theta\leq\theta_{\text{max}}$ for small $\ep>0$.
More precisely, $\theta^\ep$ converges to a solution of the bilateral obstacle problem,
\[
\min\left\{\max\left\{I\ddot\theta+\mu\dot\theta+\left\lan \dfrac{dq}{d\theta}(\theta),v(\theta_d) \right\ran,\ \theta-\theta_{\text{max}}\right\},\ \theta-\theta_{\text{min}}\right\}=0,\quad t>0
\]
as $\ep\to0$. In fact, it holds true in at least the viscosity solution sense with an appropriate initial condition. We refer to \cite{M10} for this fact.

We first show a numerical simulation result in Figure \ref{fig:sim1} in the case of Example \ref{example1}.
We take the same parameters as Example \ref{example1} (unit: [mm]) for $\kappa=30$, and $L_0=70$ [mm], $L_1=15$ [mm], 
$I=4.2\times10^{-3}$ [kg $\mathrm{m}^2$], $\mu=0.1$ [-]. The following figures are the shape of the potential energy $P(\theta)$ and a motion behavior of the joint angle $\theta$ in the case when $\theta_d=\pi/12$.

\begin{figure}[htbp]
 \begin{minipage}{0.5\hsize}
\begin{center}
\includegraphics[keepaspectratio=true,height=50mm]{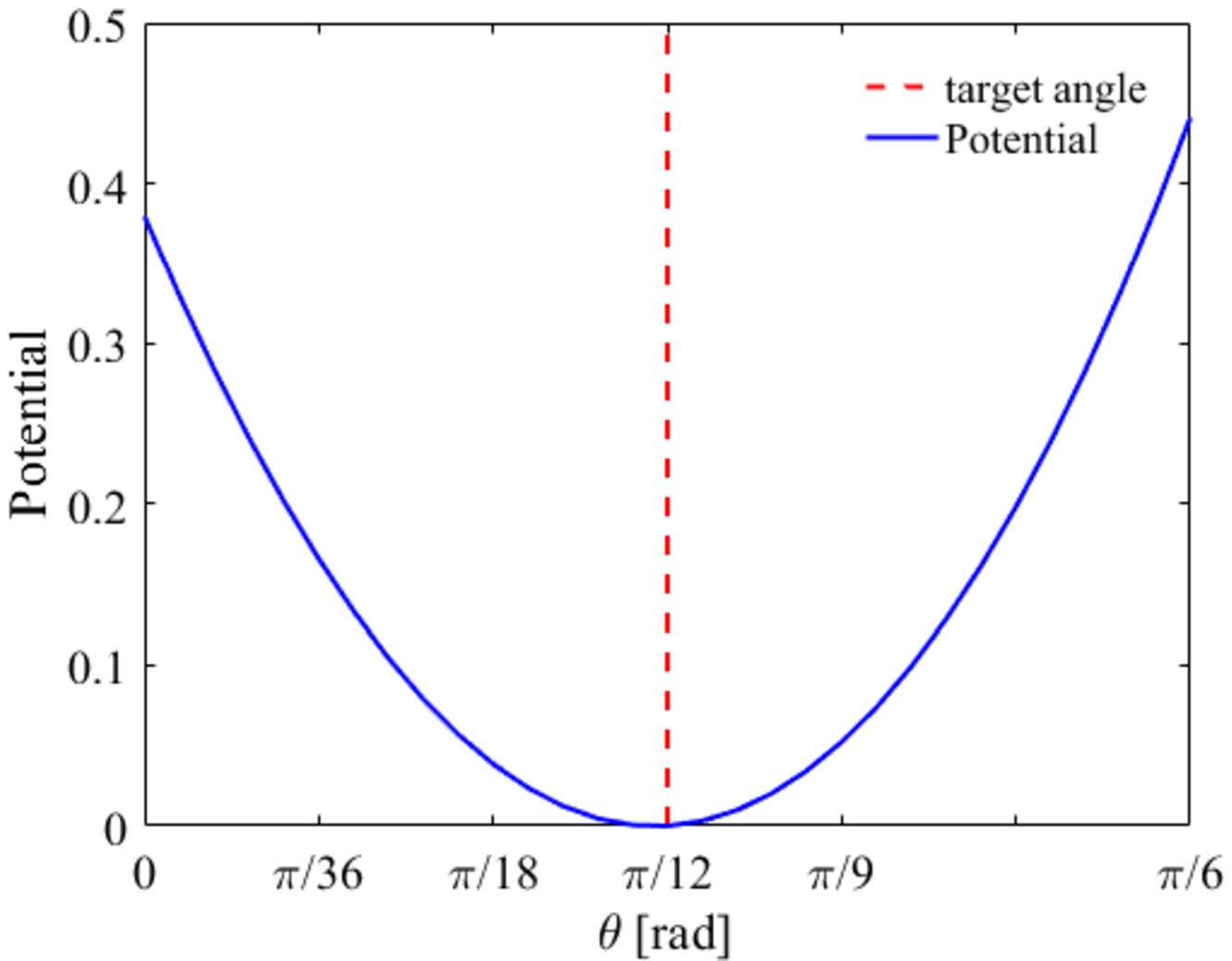}\vspace{-1.0mm}

Potential field generated by the internal muscular force balancing at $\theta_d=\pi/12$
\end{center}
 \end{minipage}
 \begin{minipage}{0.5\hsize}
\begin{center}
\includegraphics[keepaspectratio=true,height=50mm]{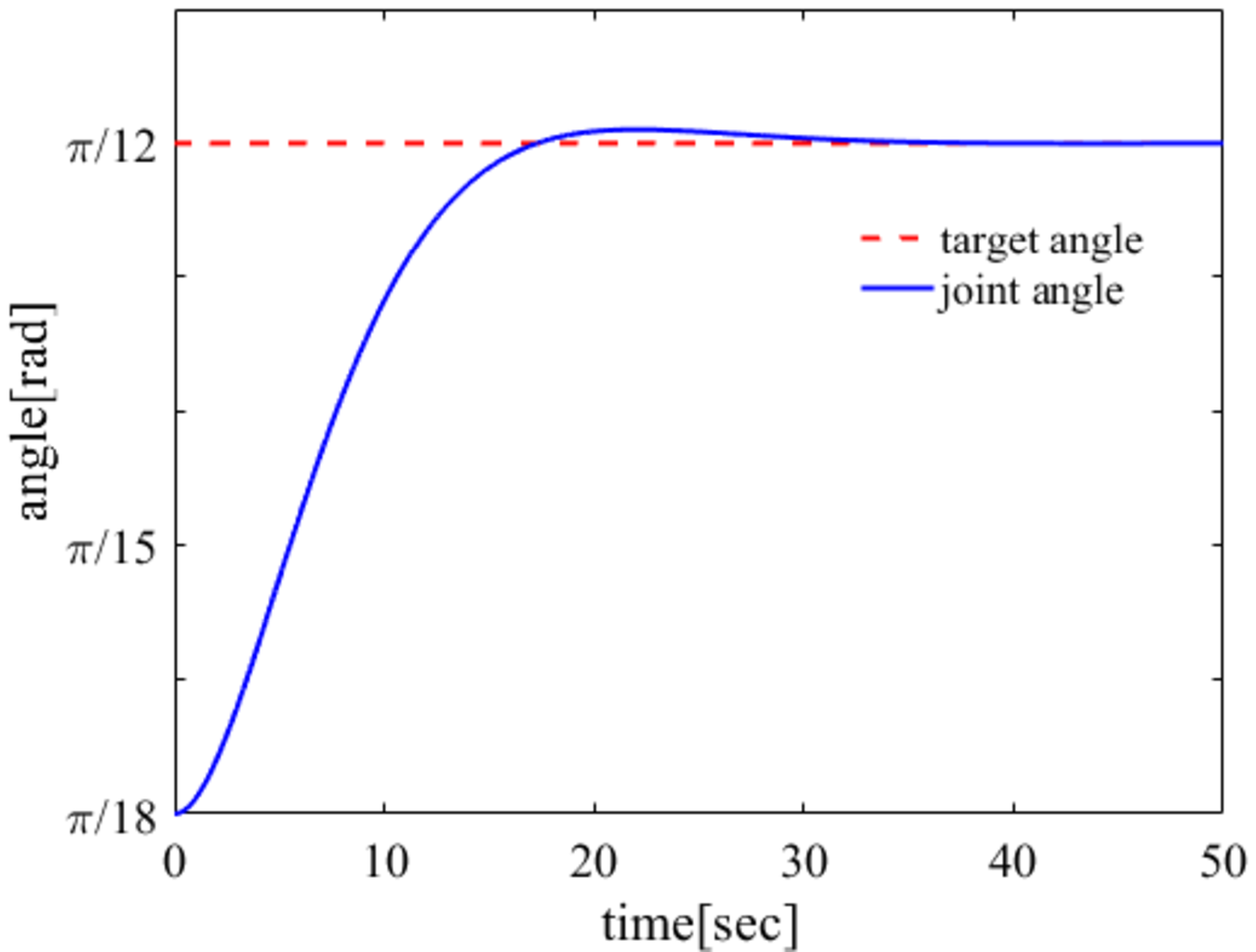}\vspace{-1.0mm}

Motion behavior of the joint angle, $\theta$, with the initial condition $(\theta(0),\dot\theta(0))=\pi/18,0)$ 
	\end{center}
 \end{minipage}
 \caption{Numerical simulation in a stable case}
 \label{fig:sim1}
\end{figure}

We next show an unstable case in Figure \ref{fig:sim2}, where the parameters are too far from a condition in Theorem \ref{thm:sc}. 
We take the same values of $L_0$, $L_1$, $I$, $\mu$ and $\theta_d$ as the above simulation. We also choose $b_i=20$ [mm], $d_i=30$ [mm], $\ell_i=30$, $r_i=15$ [mm], $s_i=25$ [mm] for $i=1,2$.

\begin{figure}[htbp]
 \begin{minipage}{0.5\hsize}
	\begin{center}
		\includegraphics[keepaspectratio=true,height=50mm]{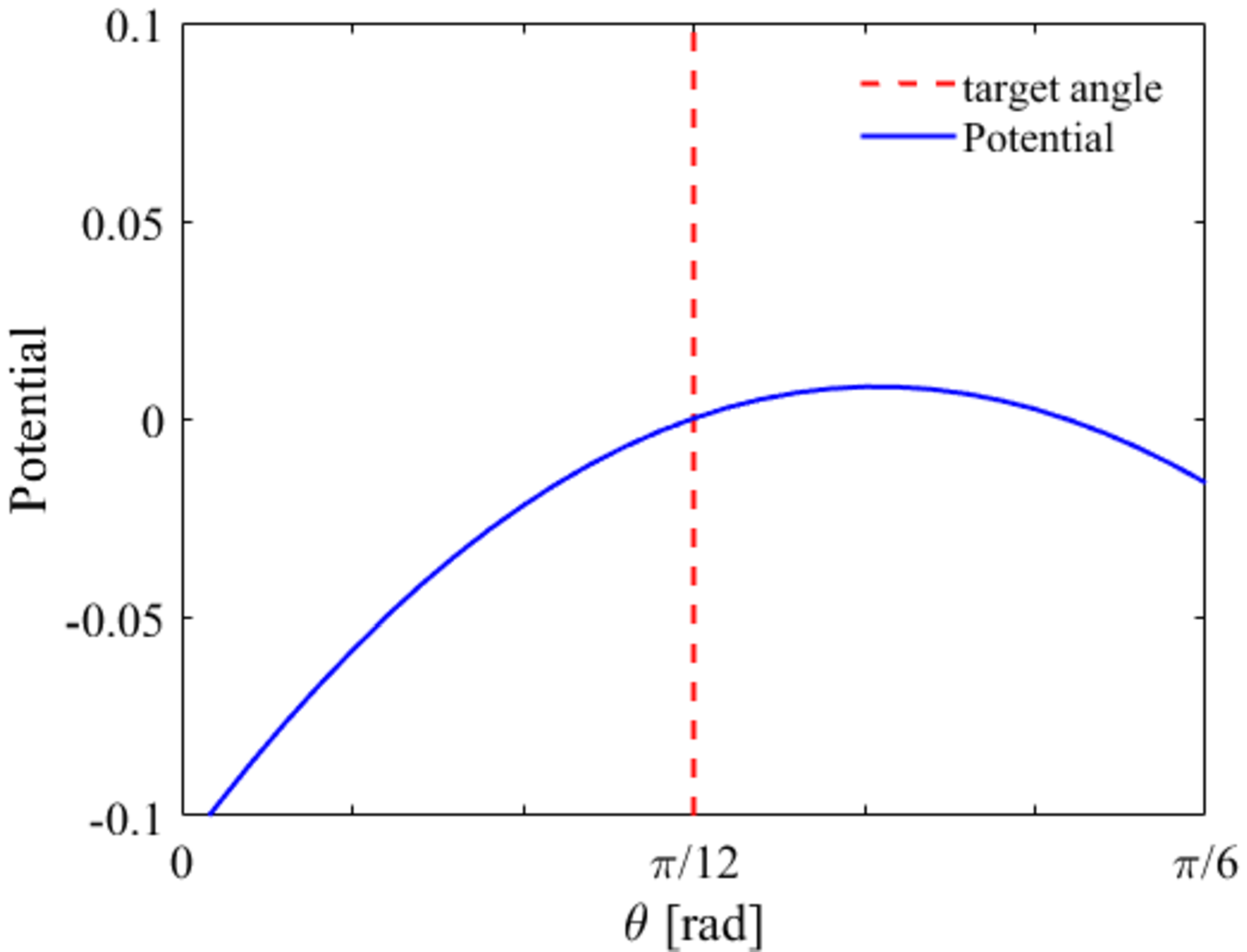}\vspace{-1.0mm}
		
	Potential field generated by the internal muscular force balancing at $\theta_d=\pi/12$
	\end{center}
 \end{minipage}
 \begin{minipage}{0.5\hsize}
	\begin{center}
		\includegraphics[keepaspectratio=true,height=50mm]{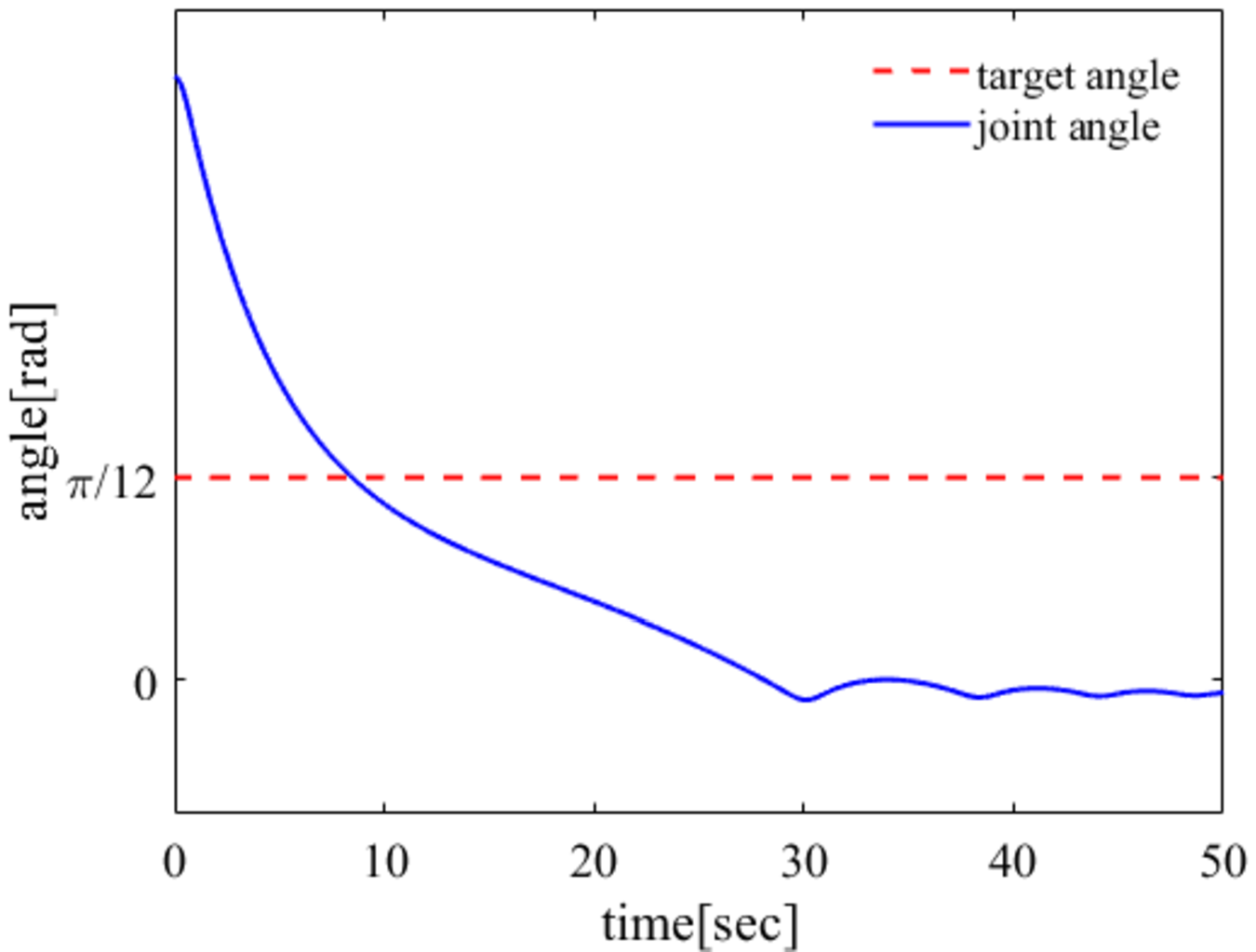}\vspace{-1.0mm}
		
		Motion behavior of the joint angle $\theta$ 
	\end{center}
 \end{minipage}
 \caption{Numerical simulation in an unstable case}
 \label{fig:sim2}
\end{figure}

\subsection{Experimental results}

We  verify a sufficient condition obtained in Theorem \ref{thm:sc} through an experiment with a real 1-link-2-muscle musculoskeletal system with routing points that is mechanically made (Figure \ref{fig:device}). 

\begin{figure}[!h]
\begin{minipage}{0.5\hsize}
\begin{center}
\includegraphics[keepaspectratio=true,height=50mm]{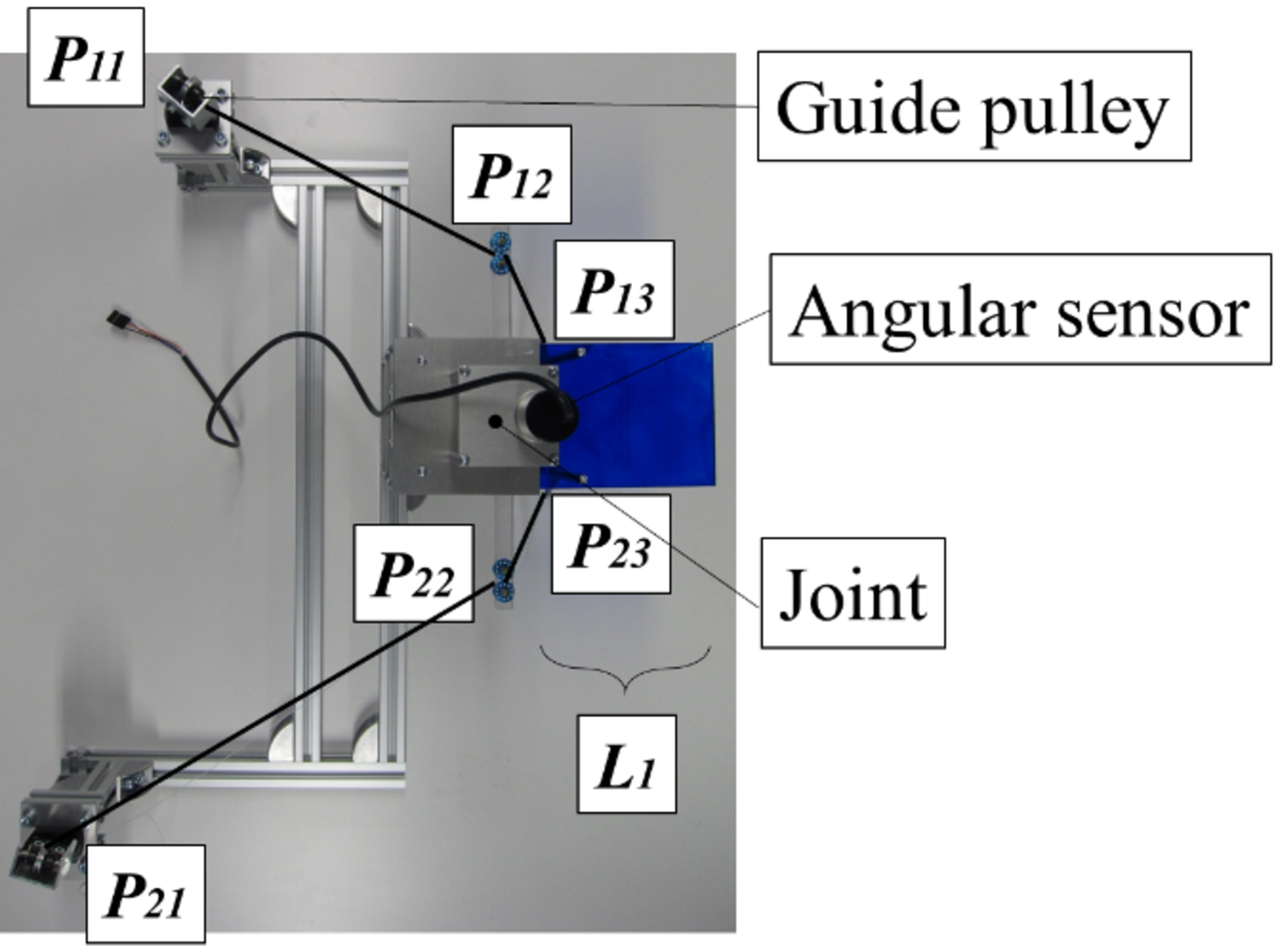}\vspace{-1.0mm}

Device
\end{center}
 \end{minipage}
 \begin{minipage}{0.5\hsize}
\begin{center}
\includegraphics[keepaspectratio=true,height=50mm]{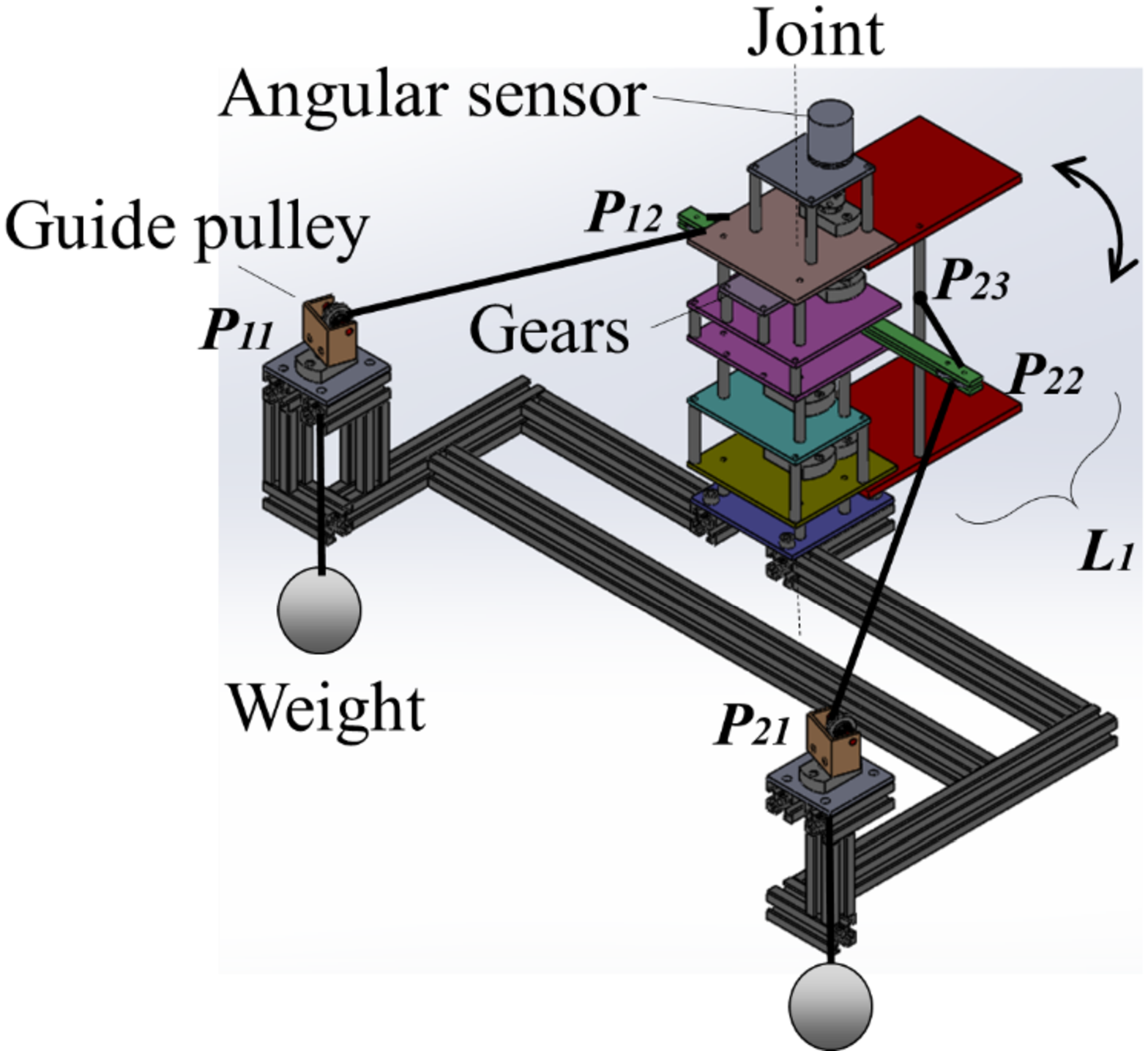}\vspace{-1.0mm}

Schematic diagram
\end{center}
 \end{minipage}
 \caption{Experiment device}
 \label{fig:device}
\end{figure}

The real system has a joint, and two fishing lines instead of muscles.
The link moves in a two-dimensional plane.
Thus, we can ignore gravity effects.
A rotary encoder is installed to measure the joint angle in the joint.
For the target system shown in Figure \ref{fig:target_system}, the rotation angle of the virtual links $\ell_1$ and $\ell_2$
 depends on the joint angle.
To accomplish this for the real system, we use several gears and replicate the virtual links.
The fishing lines pass to contact points $P_{11},P_{21}$ with the swing guide pulleys from the endpoints $P_{13},P_{23}$ of the link through the routing points $P_{12},P_{22}$, respectively.
The small pulleys of routing points $P_{12},P_{22}$ make the fishing lines smooth.
A weight at an endpoint of a fishing line generates the fishing line's tension.

We observe the motion behavior of the joint angle when we provide the real system the step input $v(\theta_d)=(v_{d1},v_{d2})$ balancing at a target angle $\theta_d$.
By changing the placement
of the swing pulleys, we can replace parameters $(d_1,d_2,b_1,b_2)$ by another.
However, $(s_1,s_2,r_1,r_2,\ell_1,\ell_2)$ cannot be changed.
We arrange the real systems in a stable case and an unstable case by replacing $(d_1,d_2)$.

Finally, we give motion behaviors of the joint angle through an experiment in a stable case based on Example \ref{example1} and an unstable case, as shown in Figure \ref{fig:exp1}.
We take the initial angle $\theta(0)=-5\pi/18$ [rad], the initial angular velocity $\dot\theta(0)=0$ [rad/s] and the target angle $\theta_d=\pi/12$ [rad]. 
Letting $\kappa\fallingdotseq 140/\sqrt2$ in Example \ref{example1}, we choose the parameters (Table \ref{table:exppara}).

\begin{figure}[h!]
\begin{minipage}{0.5\hsize}
\begin{center}
\includegraphics[keepaspectratio=true,height=50mm]{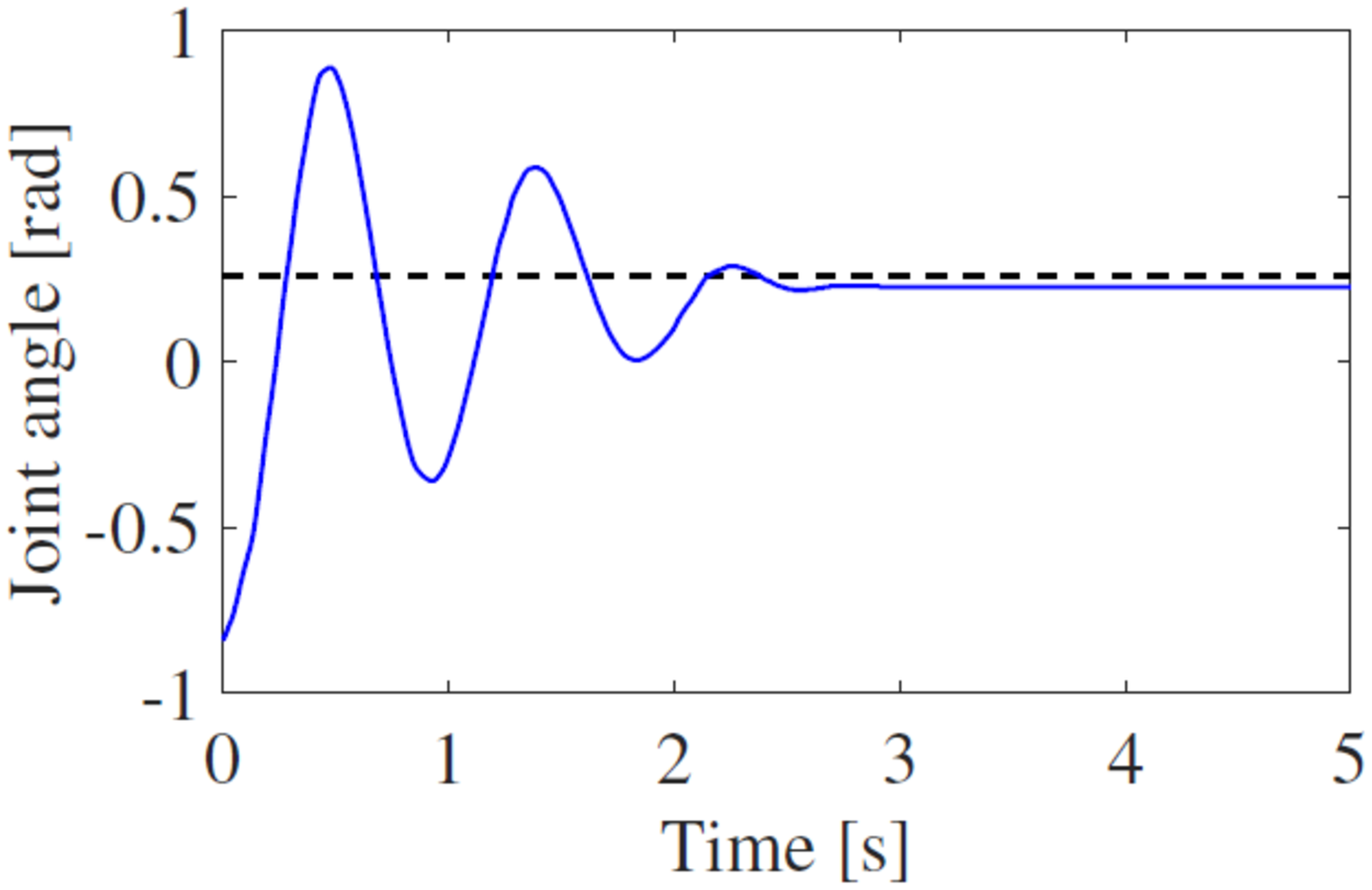}\vspace{-1.0mm}

Stable case
\end{center}
 \end{minipage}
 \begin{minipage}{0.5\hsize}
\begin{center}
\includegraphics[keepaspectratio=true,height=50mm]{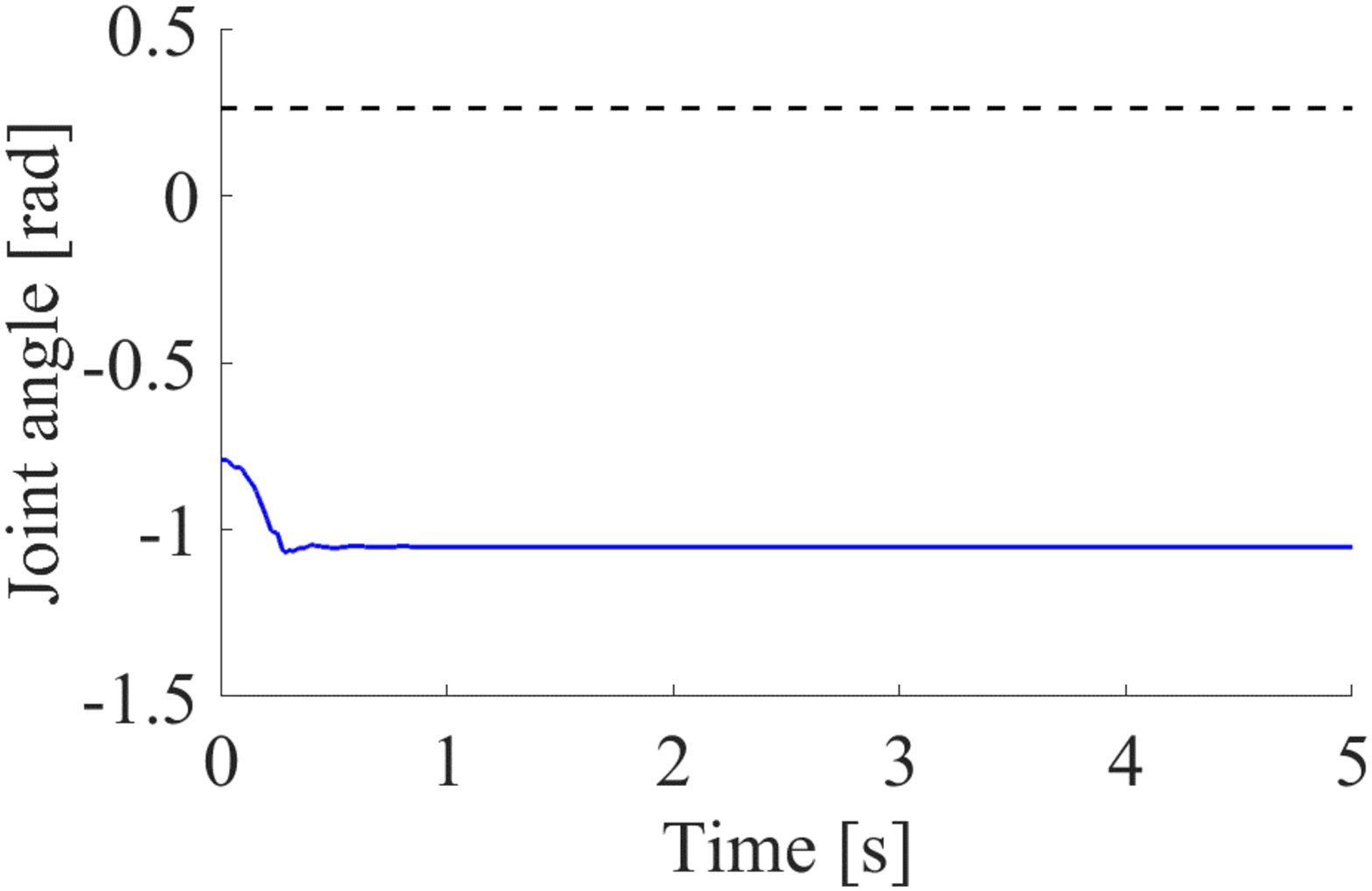}\vspace{-1.0mm}

Unstable case
\end{center}
 \end{minipage}
 \caption{Motion behavior of the joint angle}
 \label{fig:exp1}
\end{figure}

\begin{center}
\begin{table}[!h]
\begin{minipage}{0.3\hsize}
\centering
  \begin{tabular}{|c|r|} \hline
   $L_0$ & 285.0 [mm]\\ \hline 
   $L_1$ & 110.0 [mm]\\ \hline 
   $b_1$ & 87.0 [mm]\\ \hline 
   $b_2$ & 5.0 [mm]\\ \hline 
   $\ell_1$, $\ell_2$ & 99.0 [mm]\\ \hline 
   $r_1$, $r_2$ & 35.0 [mm]\\ \hline 
   $s_1$, $s_2$ & 35.0 [mm]\\ \hline 
  \end{tabular}\\[10pt]
  
  Common parameters
  \end{minipage}
  \begin{minipage}{0.3\hsize}
  \centering
  \begin{tabular}{|c|r|} \hline
   $d_1$ & 198.0 [mm]\\ \hline 
   $d_2$ & 280.0 [mm]\\ \hline 
   $v_{d1}$ & 7.84 [N]\\ \hline 
   $v_{d2}$ & 7.05 [N]\\ \hline 
   \end{tabular}\\[10pt]
   
     Stable case
     \end{minipage}
     \begin{minipage}{0.3\hsize}
       \centering
  \begin{tabular}{|c|r|} \hline
   $d_1$ & 15.0 [mm]\\ \hline 
   $d_2$ & 15.0 [mm]\\ \hline 
    $v_{d1}$ & 7.84 [N]\\ \hline 
   $v_{d2}$ & 7.63 [N]\\ \hline 
   \end{tabular}\\[10pt]
   
Unstable case
  \end{minipage}
   \caption{Experiment parameters}
    \label{table:exppara}
\end{table}
\end{center}

\end{document}